\documentclass[11pt]{amsart}
\usepackage{amsmath,amssymb,latexsym,soul,cite,mathrsfs}

\usepackage{color,enumitem,graphicx}
\usepackage[colorlinks=true,urlcolor=blue,
citecolor=red,linkcolor=blue,linktocpage,pdfpagelabels,
bookmarksnumbered,bookmarksopen]{hyperref}
\usepackage[english]{babel}

\usepackage[left=2.7cm,right=2.7cm,top=3.2cm,bottom=3.2cm]{geometry}

\usepackage[hyperpageref]{backref}

\numberwithin{equation}{section}

\newtheorem{theorem}{Theorem}[section]
\newtheorem{lemma}[theorem]{Lemma}

\newtheorem{proposition}[theorem]{Proposition}

\newcommand{\R}{{\mathbb R}}
\renewcommand{\S}{{\mathbb S}}

\newcommand{\eps}{\varepsilon}
\renewcommand{\epsilon}{\varepsilon}

\renewcommand{\theta}{{\vartheta}}

\renewcommand{\rightarrow}{\to}

\newcommand{\ud}{\mathrm{d}}
\newcommand{\N}{\mathbb{N}}
\newcommand{\dive}{\mathrm{div}}

\title[Fractional problems with exponential growth on $\R$]{Nonautonomous fractional problems \\ with exponential growth}

\author[J.M.\ do \'O]{Jo\~ao Marcos do \'O}
\author[O.H.\ Miyagaki]{Ol\'{i}mpio H.\ Miyagaki}
\author[M. Squassina]{Marco Squassina}

\address[J.M. do \'O]{Department of Mathematics,
Federal University of Para\'{\i}ba
\newline\indent
58051-900, Jo\~ao Pessoa-PB, Brazil}
\email{\href{mailto:jmbo@pq.cnpq.br}{jmbo@pq.cnpq.br}}

\address[O.\ Miyagaki]{Department of Mathematics, 
 Federal University of Juiz de Fora
\newline\indent 
36036-330  Juiz de Fora, Minas Gerais, Brazil}
\email{\href{mailto:olimpio@ufv.br}{olimpio@ufv.br}}

\address[M.\ Squassina]{Dipartimento di Informatica 
Universit\`a degli Studi di Verona,
\newline\indent
C\'a Vignal 2, Strada Le Grazie 15, I-37134 Verona, Italy}
\email{\href{mailto:marco.squassina@univr.it}{marco.squassina@univr.it}}

\thanks{Research supported in part by INCTmat/MCT/Brazil.\ J.M.\ do \'O was supported by CNPq, CAPES/Brazil, 
M.\ Squassina was supported by MIUR project Variational and Topological Methods in the Study of Nonlinear Phenomena, O.H.\ Miyagaki
 was partially supported by CNPq/Brazil and  CAPES/Brazil (Proc 2531/14-3). The paper was completed while the second author was visiting the
Department of Mathematics of Rutgers University, whose hospitality he gratefully
 acknowledges.\ He would also like to express his gratitude to Prof.\ H.\ Brezis.}
 

\begin{document}

\begin{abstract}
We study a class of nonlinear nonautonomous nonlocal equations with 
subcritical and critical exponential nonlinearity. The involved potential can vanish
at infinity.
\end{abstract}
\maketitle




\section{Introduction and main results}
\noindent
We consider existence of positive solutions for the following class of equations 
\begin{equation}
\label{PS}
(-\Delta)^{1/2} u + u = K(x) g(u) \quad\,\, \text{in $\R$}.
\end{equation}
Here   $(-\Delta)^{1/2} $ stands for the $1/2$-Laplacian, 
$K:\R\to\R$ is a positive function and $g$ is a continuous function with exponential subcritical or  critical growth in the sense of 
the Trudinger-Moser embedding due to Ozawa \cite{Ozawa}.\
Recently, a great attention has been focused on the study of nonlocal operators.
These arise in thin obstacle problems, optimization, finance, phase
transitions, stratified materials, anomalous diffusion, crystal dislocation, soft thin films, semipermeable membranes, flame
propagation, conservation laws, water waves, etc. \cite{nezza}. 
The fractional  laplacian $(-\Delta)^s$ for  $s\in(0,1)$ of a function $u:\R \to \R$ is  defined  by  
${\mathcal F}((-\Delta)^{s}u)(\xi)=|\xi|^{2s}{\mathcal F}(u)(\xi),$ 
 where ${\mathcal F}$ is the Fourier transform. Since the problem is set on the whole space
 one has to tackle compactness issues, which can be overcome by considering suitable
 assumptions of the vanishing behaviour of $K$ at infinity.
Recently the problem in $\R^N$ with $N>2s$, $s\in (0,1)$, $2^*_s=2N/(N-2s)$,
\begin{equation*}
(-\Delta)^s u + V(x)u = K(x)g(u) + \lambda |u|^{2^{*}_{s}-2}u \quad\,\, \text{in $\R^N$}.
\end{equation*}
where $g$ has subcritical growth has been investigated in \cite{JMO} inspired by some arguments of \cite{AS2012}. 
The aim of this paper is to extend the achievements of \cite{JMO} to cover the case where the nonlinearity is allowed to grow at an
exponential rate. As is pointed out in \cite{Antonio}, nonlocal problems with linear fractional diffusion involving exponential 
growth should be set in $\R$. In that manuscript the authors prove existence results for problems
involving critical and subcritical exponential growth nonlinearities
and $1/2$-Laplacian in a bounded domain. The main ingredient is  the Trudinger-Moser type inequality \cite{Ozawa} (see Proposition~\ref{moser}).
For related problems involving Moser-Trudinger embeddings
we would like to mention 
the celebrated works \cite{M,T} as well as \cite{DMR,adimur,adimur2,Def-cpam,carl,adams,DOR} and the references therein.
As known, Caffarelli and Silvestre \cite{caffarelli} developed a local interpretation of the fractional
Laplacian by considering a Neumann type operator in 
$\R^{N+1}_{+}=\{(x, t) \in \R^{N+1} : t > 0\}$. A similar extension, for nonlocal problems on bounded domain with
the zero Dirichlet boundary condition,  was also studied, see \cite{cabretan}.
The space $\dot H^{1/2}(\R)$
is the completion of $C^{\infty}_{0}(\R)$ under 
\begin{equation*}
[u]_{{1/2}}:=\Big(\int_{\R}|\xi||{\mathcal F} u|^2 \,\ud\xi\Big)^{1/2}=\Big(\int_{\R}|(-\Delta)^{1/4}u|^2 \,\ud x\Big)^{1/2}\!\!\!\!,
\end{equation*}
while $H^{1/2}(\R)$ is the Hilbert space of $u\in L^2(\R)$ such that $[u]_{H^{1/2}}<\infty$, endowed with the norm
$$
\|u\|_{{1/2}}=\big(\|u\|_{L^2}^2+[u]_{{1/2}}^2\big)^{1/2}.
$$
The space $X^{1}(\R^{2}_{+})$ is defined as the completion of $C^{\infty}_{0}(\overline{\R^{2}_{+}})$  under the semi-norm
\begin{equation*}
\|w\|_{X^{1}}:=\Big(\int_{ \R^{2}_{+}}|\nabla w|^2 \,\ud x\ud y\Big)^{1/2}.
\end{equation*}
For a function $u\in \dot H^{1/2}(\R)$, the solution $w \in X^{1}(\R^{2}_{+})$ to 
 \begin{equation}
 \left\{ \begin{array}{rcl}
 -\dive (\nabla w)=0 & \mbox{in}&  \R^{2}_{+}\noindent\\
  w=u& \mbox{on}& \R \times\{0\}\noindent
 \end{array}\right. 
 \end{equation}
is called harmonic  extension $w=E_{1/2}(u)$ of $u$ and it is proved in \cite{caffarelli,colorado} that, up to some constant, 
$$
\lim_{y \to 0^+} \frac{\partial w}{\partial y}(x,y)=-(-\Delta)^{1/2}u(x).
$$
Also, up to a constant, $[u]_{{1/2}}=\|w\|_{X^{1}}$, see  \cite{colorado}.
Our problem \eqref{PS} will be studied in the  half-space, 
 \begin{equation}\label{NPS}
 \left\{ \begin{array}{rcl}
 -\dive ( \nabla w)=0 & \mbox{in}&  \R^{2}_{+}\noindent\\
  -\frac{\partial w}{\partial \nu}=-u +K(x)g(u) & \mbox{on}& \R \times\{0\},\noindent
 \end{array}\right. 
 \end{equation}
where  $\frac{\partial w}{\partial \nu}=\lim_{y \to 0^+} \frac{\partial w}{\partial y}(x,y).$ 
We look for positive solutions in the Hilbert space $E$ defined by
$$
E:=\Big\{ w \in X^{1}(\R^{2}_{+}): \ \int_{\R}w(x,0)^2 \,\ud x < \infty\Big\},
$$
endowed with the norm
$$
\|w\|:= \Big(\int_{ \R^{2}_{+}}|\nabla w|^2 \,\ud x\ud y + \int_{\R}w(x,0)^2 \,\ud x\Big)^{1/2}.
$$
Consider now the energy functional $J:E\to\R$ associated to \eqref{NPS}  given by
\begin{equation}\label{functional}
J(w):=\frac{1}{2}\|w\|^2 - \int_{\mathbb{R}} K(x) G(w(x,0)  \,\ud x,\quad\,\,\,\,   G(s):=\int_{0}^{s}g(t)\,\ud t,
\end{equation}
which, under suitable assumptions, is $C^1$  (see Proposition~\ref{C-uno}) and, for all $w,v\in E$,
\begin{equation}
\label{derivative}
 J'(w)(v) =\int_{ \R^{2}_{+}}\nabla w \cdot\nabla v  \,\ud x \ud y + \int_{\R}w(x,0)v(x,0)  \,\ud x - \int_{\mathbb{R}} K(x) g(w(x,0))v(x,0)  \,\ud x.
\end{equation}
We now formulate assumptions for $K$ and $g$ in order to be able to solve~\eqref{PS}.
\vskip4pt
\noindent
$\bullet$ {\em Assumption on $K$.} 
We assume $K\in L^\infty(\R)\cap C(\R)$. Furthermore, if $\{A_n\}$ is a sequence of Borel sets of $\R$ with $|A_n|\leq R$ for some $R>0$, 
\begin{equation}
\label{decay-K}
\lim_{r\to \infty} \int_{A_n \cap B^{c}_{r}(0)} K(x) \,\ud x =0,\quad\text{uniformly with respect to $n \in \N$.}
\end{equation}
$\bullet$ {\em Assumptions on $g$ -- subcritical case} 
\begin{description}
\item[(g1) (behaviour at zero)]  $g:\R\to\R^+$ is continuous with $g=0$ on $\R^-$ and
$$
\limsup_{s\to 0^+}\frac{g(s)}{s}=0.
$$
\item[(g2) (subcritical growth)] it holds
$$
\limsup_{s\to +\infty} \frac{g(s)}{e^{\alpha s^2}-1}=0, \quad \mbox{for all $\alpha >0$}.
$$
\item[(g3) (super-quadraticity)]  $\frac{g(s)}{s}$ is non-decreasing in $\R^+$ and 
$$
\limsup_{s\to +\infty} \frac{G(s)}{s^{2}}=+\infty.
$$
\end{description}

\noindent
Under assumption~\eqref{decay-K} on $K$, we have the following

\begin{theorem}
Assume {\rm (g1)-(g3)}. Then \eqref{PS} admits a positive solution $u\in H^{1/2}(\R)$.
\label{mainsub}
\end{theorem}

\noindent
$\bullet$ {\em Assumptions on $g$ -- critical case} 
\begin{description}
\item[(g2)$'$ (critical growth)] there exists $\omega\in (0,\pi]$ and $\alpha_0\in (0,\omega]$
\begin{align*}
\limsup_{s\to +\infty} \frac{g(s)}{e^{\alpha s^2}-1}&=0, \quad \mbox{for all $\alpha >\alpha_0$}, \\
\limsup_{s\to +\infty} \frac{g(s)}{e^{\alpha s^2}-1}&=+\infty, \quad\mbox{for all $\alpha <\alpha_0$}.
\end{align*}
\item[(g3)$'$ (super-quadraticity)] $\frac{g(s)}{s}$ is non-decreasing in $\R^+$ and there are $q>2$ and $C_q>0$ with
$$
G(s)\geq C_q s^q, \quad \text{for all $s\in\R^+$}.
$$
\item[(AR) (Ambrosetti-Rabinowitz)] there exists $\theta>2$ such that
\begin{equation}
 \theta G(s)\leq sg(s), \quad\text{ for all $s\in\R^+$}.
\tag{AR}
\end{equation}
\end{description}

\vskip2pt
\noindent
Under assumption~\eqref{decay-K} on $K$, we also have the following

\begin{theorem}
\label{main}
Assume {\rm (g1)-(g2)$'$-(g3)$'$} and (AR).
Then \eqref{PS} has a positive solution $u\in H^{1/2}(\R)$ provided that the constant $C_q$ in condition {\rm (g3)$'$}
is sufficiently large.
\end{theorem}

\noindent
The above results extend the existence results obtained in \cite{Antonio} in the case where the problem
is set on the whole $\R$ and compactness issues have to be tackled. Also, they constitute an extension 
to the results of \cite{JMO} to the case where the nonlinearity is allowed for an exponential growth,
critical or subcritical with respect to the Trudinger-Moser inequality \eqref{I1}. As potentials $K$
satisfying \eqref{decay-K}, one can consider $K$s with $K(x)\to 0$ as $|x|\to\infty$. As examples of nonlinearities
satisfying the above assumptions, define $g:\R\to\R^+$ by setting
$g(t)=0$ for all $t\leq 0$ and 
\[
g(t)=\left\{
	\begin{array}{ll}
		t^q  & \mbox{if } 0\leq t\leq 1, \\
	    t^q e^{t^r-1} & \mbox{if } t\geq 1,
	\end{array}
\right.
\qquad 1<r<2,\,\,\,\, q>1,
\]
This function satisfies {\rm (g1)-(g3)}. Define $g:\R\to\R^+$ by setting $g(t)=0$ for all $t\leq 0$ and 
\[g(t)=C_q\left\{
	\begin{array}{ll}
		 t^q  & \mbox{if } 0\leq t\leq 1, \\
	     t^q e^{\alpha_0 (t^2-1)} & \mbox{if } t>1,
	\end{array}
	\quad q>2,
\right.
\]
where $\alpha_0\in (0,\omega]$ and $C_q>0$ is sufficiently large. This map satisfies {\rm (g1)-(g2)$'$-(g3)$'$} and (AR).

\smallskip
\section{Preliminary results}

\noindent
In this section we provide some preliminary stuff. Consider the weighted Banach space
$$
L^{p}_{K}(\R)=\Big\{ u:\R \to \R\,\, \mbox{measurable:}\ \int_{\R}K(x)|u|^{p} \, \ud x < \infty\Big\},\quad p\in (1,\infty),
$$
endowed with the norm
$$
\|u\|_{L^{p}_{K}}=\Big(\int_{\R}K(x)|u|^{p} \, \ud x\Big)^{1/p}.
$$
\noindent
The first result, is a compact injection for the space $E$.

\begin{proposition}
\label{converge}
$E$ is compactly embedded  into $L^q_{K}(\R)$  for all $q\in (2,\infty)$.
\end{proposition}
\begin{proof}
Let $q>2$, $r>q$ and $\epsilon >0$. Then, there exist $0<s_0(\eps)<s_1(\eps)$, a positive constant 
$C(\eps)$ and $C_0$ depending only on $K$, such that
\begin{equation}
\label{2.4}
 K(x)|s|^q\leq \epsilon C_0 (|s|^2 + |s|^{r})+C(\eps)K(x)\chi_{[s_0(\eps),s_1(\eps)]}(|s|)|s|^{q}, \quad  x,s\in \R.
\end{equation}
Therefore we obtain, for every $w \in E$ and $r>0$,
\begin{equation}\label{2.5}
 \int_{B^{c}_{r}(0)} K(x)|w(x,0)|^q \,\ud x \leq \epsilon Q(w)+ C(\eps)s_1(\eps)^{q} \int_{A_\eps\cap B^{c}_{r}(0)} K(x) \,\ud x,
\end{equation}
where we have set
\begin{equation}
\label{Q-A}
Q(w):=   C_0 \|w(\cdot,0)\|^2_{L^2} + C_0\|w(\cdot,0)\|^{r}_{L^r},\quad 
A_\eps :=  \left\{x\in \R:  s_0(\eps) \leq |w(x,0)|\leq s_1(\eps)\right\}.
\end{equation}
If $(w_n)\subset E$ is such that $w_n \rightharpoonup w$ weakly in $E$ for some $w\in E$, there exists $M>0$ such that 
\begin{equation}
\label{boundednesses}
\begin{aligned}
 \int_{\R^{2}_{+}} |\nabla w_n|^2 \,\ud x \ud y + \int_{\R}|w_n(x,0)|^2 \,\ud x \leq & M, \\ 
  \int_{\R}|w_n(x,0)|^{r} \, \ud x\leq & M,\quad \text{for all $r\geq 2$}.
  \end{aligned}
\end{equation}
The second inequality is due to the continuous injection of $H^{1/2}(\R)$ in an
arbitrary $L^r(\R)$ space with $r\geq 2$, see \cite[Theorem 6.9]{nezza}. Hence $Q(w_n)$ is bounded. On the other hand, if  
$$
A^n_\eps:=\big\{x\in\R:s_0(\eps) \leq |w_n(x,0)|\leq s_1(\eps)\big\},
$$  
we get 
$$
s_0(\eps)^{q}|A^n_\eps|\leq \int_{A^n_\eps}|w_n(x,0)|^{q} \,\ud x \leq \int_{\R^N}|w_n(x,0)|^{q} \,\ud x\leq M, \quad \text{for all $n\in \N$}. 
$$
which implies that $\sup_{n\in \N} |A^n_\eps|<+\infty$.
Then, in light of \eqref{decay-K}, there exists $r(\eps)>0$ such that
\begin{equation}\label{2.6}
 \int_{A^n_\eps\cap B^{c}_{r(\eps)}(0)} K(x) \,\ud x <\frac{\epsilon}{C(\eps) s_1(\eps)^{q}},\quad \text{for all $n\in \N$}. 
\end{equation}
Whence, in light of \eqref{2.5}, we conclude
\begin{equation}
\label{2.7}
 \int_{B^{c}_{r(\eps)}(0)} K(x)|w_n(x,0)|^q \,\ud x \leq  (2C_0M+1)\eps. 
\end{equation}
By the fractional compact embedding \cite[Theorem 7.1]{nezza}, we have
\begin{equation}\label{2.8}
 \lim_n \int_{B_{r(\eps)}(0)} K(x)|w_n(x,0)|^q \,\ud x=\int_{B_{r(\eps)}(0)} K(x)|w(x,0)|^q \,\ud x.
\end{equation}
Combining \eqref{2.7}-\eqref{2.8}, yields 
$$
\lim_{n}\int_{\R} K(x)|w_n(x,0)|^q \,\ud x=\int_{\R} K(x)|w(x,0)|^q \,\ud x.
$$
This concludes the proof.
\end{proof}

\noindent
Let us now recall the Trudinger Moser type inequality of \cite{Ozawa}.

\begin{proposition}\label{moser}
 There exists $0 < \omega \leq \pi$ such that, for all $\alpha \in (0, \omega)$, there exists $H_{\alpha}>0$ with
\begin{equation}\label{I1}
\int_{\R} (e^{\alpha u^2}-1) \, \ud x \leq H_{\alpha}\|u\|^{2}_{L^2},
\end{equation}
for all $u\in H^{1/2}(\R)$ with $\|(-\Delta)^{1/4} u\|^2_{L^2}\leq 1$.
\end{proposition}

\noindent
Next, we state a useful Trudinger-Moser type bound for bounded sequences of $E$.

\begin{lemma}
\label{boundmoser}
Let $(w_n)\subset E$ be a bounded sequence and set $\sup_{n\in\N}\|w_n\|=M$. Then
$$
\sup_{n\in\N}\int_{\R} (e^{\alpha w_n(x,0)^2}-1) \,\ud x<\infty,\quad\text{for every $0<\alpha<\frac{\omega}{M^2}$};
$$
In particular, if $M\in (0,1)$, there exists $\alpha_M>\omega$ such that
$$
\sup_{n\in\N}\int_{\R} (e^{\alpha_M w_n(x,0)^2}-1) \,\ud x<\infty.
$$
\end{lemma}
\begin{proof}
Let $0<\alpha M^2<\omega$. Then, setting $u_n(x)=w_n(x,0)$, by virtue of Proposition~\ref{moser}, we have
\begin{equation}
\label{chain-exp}
\int_{\R} \big(e^{\alpha u_n^2}-1\big) \,\ud x\leq \int_\R \big(e^{\alpha M^2 \big(\frac{u_n}{\|w_n\|}\big)^2}-1\big) \,\ud x 
\leq H_{\alpha M^2} \frac{\|u_n\|_{L^2}^2}{\|w_n\|^2}\leq H_{\alpha M^2},
\end{equation}
since $\|(-\Delta)^{1/4} u_n\|w_n\|^{-1}\|^2_{L^2}=\|(-\Delta)^{1/4} u_n\|^2_{L^2}/\|w_n\|^2
=\|w_n\|_{X^1}^2/\|w_n\|^2\leq 1$. Concerning the last assertion, there exists $\alpha_M>\omega$
with $\alpha_M M^2<\omega$ and the conclusion follows.
\end{proof}

\begin{lemma}
\label{exp-conv}
Let $\alpha>0$ and let $(w_n)\subset E$ be such that $w_n\to w$ strongly in $E$. Then
$$
\lim_n\int_\R \big(e^{\alpha w_n(x,0)^2}-1\big) \,\ud x=\int_\R  \big(e^{\alpha w(x,0)^2}-1 \big) \,\ud x.
$$
\end{lemma}
\begin{proof}
By applying Lagrange's theorem to the function $s\mapsto e^{\alpha s^2}$, we get
\begin{align*}
& \big| \big( e^{\alpha w_n(x,0)^2}-1\big)- \big(e^{\alpha w(x,0)^2}-1\big) \big| \\
& \leq 2\alpha(|w_n(x,0)-w(x,0)|+|w(x,0)|)e^{2\alpha|w_n(x,0)-w(x,0)|^2}e^{2\alpha|w(x,0)|^2}|w_n(x,0)-w(x,0)|.
\end{align*}
The right-hand side splits into several terms. We shall handle one of them, namely
$$
(|w_n(x,0)-w(x,0)|+|w(x,0)|)(e^{2\alpha|w_n(x,0)-w(x,0)|^2}-1)(e^{2\alpha|w(x,0)|^2}-1)|w_n(x,0)-w(x,0)|,
$$
since the other terms can be handled in a similar fashion. Then one applies H\"older inequality with four
terms with exponents $r_1,r_4\geq 2$ and $r_2,r_3>1$ such that $1/r_1+1/r_2+1/r_3+1/r_4=1$.
Recall that $(e^{x}-1)^r\leq (e^{rx}-1)$ holds for $r>1$ and $x\geq 0$. For the first term, 
$\|w_n-w\|_{L^{r_1}}+\|w\|_{L^{r_1}}\leq C$ by the continuous Sobolev embedding in any $L^r$ space with $r\geq 2$.
For the second term, since $\|w_n-w\|\to 0$, one can apply Lemma~\ref{boundmoser} (this is the key point of the proof) and deduce
$$
\int_\R \big(e^{2r_2\alpha |w_n(x,0)-w(x,0)|^2}-1\big) \,\ud x\leq C.
$$
For the third term we have 
$$
\int_\R \big(e^{2r_3\alpha |w(x,0)|^2}-1\big) \,\ud x<\infty.
$$
Here we used that $e^{u^2}-1\in L^1(\R)$ for $u\in H^{1/2}(\R)$, see the argument in \cite[Proposition 2.5]{Antonio}.
Finally the last term is estimated with $\|w_n-w\|_{L^{r_4}}$, which goes to zero and conclude the proof.
\end{proof}

\noindent
The following is a straightforward application of Fatou's lemma.

\begin{lemma}
\label{byFat}
Let $f_n,g_n,h_n:\R\to\R^+$ sequences of nonnegative measurable functions. Assume that $f_n$ converges pointwisely to $0$
and that $g_n,h_n$ converge pointwisely to $g,h:\R\to\R^+$. Assume also that, for every $\eps>0$,
there exists $C(\eps)>0$ such that
$$
f_n\leq \eps g_n+C(\eps) h_n,\,\,\,\, n\in\N,
\qquad \sup_{n\in\N}\int_\R g_n \,\ud x<\infty,\quad
\lim_n\int_\R h_n \,\ud x=\int_\R h \,\ud x.
$$
Then $f_n\to 0$ in $L^1(\R)$.
\end{lemma}

\noindent
We can now state the following compactness result for the subcritical growth case.

\begin{proposition}[Compactness I -- subcritical case]
\label{convergeG1} 
Assume {\rm (g1)-(g3)}. Let  $(w_n)\subset E$ be a bounded sequence and $w_n \rightharpoonup w$ in $E$.
Then, up to a subsequence, the following facts hold:
\begin{align}
&\lim_n \int_{\mathbb{R}} K(x) G(w_n(x,0))  \,\ud x  =  \int_{\mathbb{R}} K(x) G(w(x,0))  \,\ud x ; \label{Veneto}\\
&\lim_n\int_{\mathbb{R}} K(x) w_n(x,0) g(w_n(x,0))   \,\ud x  =  \int_{\mathbb{R}} K(x) w(x,0) g(w(x,0))  \,\ud x ; \label{Toscana}\\
&\lim_n\int_{\mathbb{R}} K(x)  g(w_n(x,0))v(x,0)  \,\ud x  =  \int_{\mathbb{R}} K(x)  g(w(x,0)v(x,0))  \,\ud x , \,\,\,\,\text{for all $v \in E$} \label{Roma}.
\end{align}
\end{proposition}
\begin{proof}
Let us prove \eqref{Veneto} and \eqref{Toscana}. Let $\sup_{n\in\N}\|w_n\|=:M$.
Let us also fix $\eps>0$, $q>2$ and $0<\alpha<\omega/M^2$, according to Lemma~\ref{boundmoser}. In light of (g2), we learn that
$$
\limsup_{s\to +\infty} \frac{g(s)s}{e^{\alpha s^2}-1}=\limsup_{s\to +\infty} \frac{G(s)}{e^{\alpha s^2}-1}=0,\quad
\limsup_{s\to 0^+} \frac{g(s)s}{s^2}=\limsup_{s\to 0^+} \frac{G(s)}{s^2}=0.
$$
Then there exist $0<s_0(\eps)<s_1(\eps)$, $C(\eps)>0$  and $C_0$ depending only upon $K$, such that
\begin{align}\label{2.12}
 |K(x)G(s)|&\leq \epsilon C_0(s^2 +e^{\alpha s^2}-1)+ C(\eps) K(x)\chi_{[s_0(\eps),s_1(\eps)]}(|s|)|s|^q, \,\,\quad\text{for all $s\in \R$}, \\
\label{2.12.1}
  |K(x)g(s)s|&\leq \epsilon C_0(s^2 +e^{\alpha s^2}-1)+ C(\eps) K(x)\chi_{[s_0(\eps),s_1(\eps)]}(|s|)|s|^q, \,\,\quad\text{for all $s\in \R$}.
\end{align}
By virtue of Lemma~\ref{boundmoser} we find $E>0$ such that
\begin{equation}
\label{bound-exp}
\sup_{n\in\N}\int_{\R} (e^{\alpha w_n(x,0)^2}-1) \,\ud x\leq E,
\qquad  \int_\R |w_n(x,0)|^2 \,\ud x\leq E.
\end{equation}
Notice again that, by means of \eqref{decay-K}, there exists $r(\eps)>0$ such that
\begin{equation}\label{2.13}
 \int_{A^n_\eps \cap B^c_{r(\eps)}(0)} K(x) \,\ud x\leq \frac{\epsilon}{C(\eps)s_1(\eps)^q}, \quad\text{for all $n \in \N$}.
\end{equation}
Now, combining the above inequality with \eqref{2.12}-\eqref{2.12.1}, we have
\begin{align}\label{2.14}
& \int_{B^{c}_{r(\eps)}(0)} K(x)G(w_n(x,0)) \,\ud x\leq (2C_0E + 1)\epsilon ,\quad\text{for all $n \in \N$}, \\
\label{2.14.1}
& \int_{B^{c}_{r(\eps)}(0)} K(x)g(w_n(x,0))w_n(x,0) \,\ud x\leq (2C_0E + 1)\epsilon ,\quad\text{for all $n \in \N$}.
\end{align}
Notice that we have
\begin{align*}
|K(x)(G(w_n(x,0))-G(w(x,0))|&\leq \epsilon (w_n(x,0)^2 +e^{\alpha w_n(x,0)^2}-1 
+w(x,0)^2 +e^{\alpha w(x,0)^2}-1) \\
&+ C(\eps)(|w_n(x,0)|^q+|w(x,0)|^q).
\end{align*}
A similar estimation holds for $K(x)g(s)s$. 
Hence, by \eqref{bound-exp} and since $w_n(x,0)\to w(x,0)$ in $L^q(B_{r(\eps)(0)})$
by the compact embedding \cite[Theorem 7.1]{nezza}, Lemma~\ref{byFat} allows to conclude that
\begin{align*}
 & \lim_{n}\int_{B_{r(\eps)}(0)} K(x)
G(w_n(x,0)) \,\ud x=\int_{B_{r(\eps)}(0)} K(x)G(w(x,0)) \,\ud x, \\
& \lim_{n}\int_{B_{r(\eps)}(0)} K(x)
g(w_n(x,0))w_n(x,0) \,\ud x=\int_{B_{r(\eps)}(0)} K(x)g(w(x,0)w(x,0) \,\ud x.
\end{align*}
Combining these with \eqref{2.14}-\eqref{2.14.1} we conclude the proof.
\noindent
Let us now prove \eqref{Roma}. The sequence
$(\sqrt{K(x)} g(w_n(x,0))\chi_{\{|w_n(x,0)|\leq 1\}})$ is bounded in $L^2(\R)$ as by (g1)
$$
|\sqrt{K(x)} g(w_n(x,0))\chi_{\{|w_n(x,0)|\leq 1\}}|^2\leq C |w_n(x,0)|^2.
$$
This, by pointwise convergence, yields for every $\varphi\in L^2(\R)$
$$
\lim_k\int_{\R} \sqrt{K(x)} g(w_n(x,0))\chi_{\{|w_n(x,0)|\leq 1\}}\varphi(x) \,\ud x=\int_{\R} \sqrt{K(x)} g(w(x,0))
\chi_{\{|w(x,0)|\leq 1\}}\varphi(x) \,\ud x.
$$
Given $v\in E$, it follows $\sqrt{K(x)}v(x,0)\in L^2(\R)$, yielding 
\begin{equation*}
\lim_k\int_{\R} K(x) g(w_n(x,0))\chi_{\{|w_n(x,0)|\leq 1\}} v(x,0) \,\ud x=\int_{\R} K(x) g(w(x,0))\chi_{\{|w(x,0)|\leq 1\}} v(x,0) \,\ud x.
\end{equation*}
Moreover, by  $(g2)$,
$(K(x) g(w_n(x,0))\chi_{\{|w_n(x,0)|\geq 1\}})$ is bounded in $L^m(\R)$ by Lemma~\ref{boundmoser} as
$$
|K(x) g(w_n(x,0))\chi_{\{|w_n(x,0)|\geq 1\}}|^{m}\leq C(e^{m\alpha w_n(x,0)^2}-1),\quad \text{for $m\alpha<\omega/M^2$}.
$$
Here $m>1$ is taken close to $1$. Then, for all $v\in E\subset L^{m'}(\R)$ (notice that $m'>2$), we get
\begin{equation*}
\lim_k\int_{\R} K(x) g(w_n(x,0))\chi_{\{|w_n(x,0)|\geq 1\}} v(x,0) \,\ud x=\int_{\R} K(x) g(w(x,0))\chi_{\{|w(x,0)|\geq 1\}} v(x,0) \,\ud x.
\end{equation*}
This concludes the proof of \eqref{Roma}.
\end{proof}

\noindent
From now on, in assumption (g2)$'$, we can assume that $\alpha_0=\omega$, without loss of generality.
We can state the following for the critical growth case.

\begin{proposition}[Compactness II -- critical case]
\label{convergeG2} 
Assume {\rm (g1)-(g2)$'$-(g3)$'$}. Let  $(w_n)\subset E$ a bounded sequence and $w_n \rightharpoonup w$ in $E$ such that
$$
\sup_{n\in\N}\|w_n\|\in (0,1).
$$
Then, up to a subsequence, the following facts hold:
\begin{align}
&\lim_n \int_{\mathbb{R}} K(x) G(w_n(x,0))  \,\ud x  =  \int_{\mathbb{R}} K(x) G(w(x,0))  \,\ud x ; \label{Veneto1}\\
&\lim_n\int_{\mathbb{R}} K(x) w_n(x,0) g(w_n(x,0))   \,\ud x  =  \int_{\mathbb{R}} K(x) w(x,0) g(w(x,0))  \,\ud x ; \label{Toscana1}\\
&\lim_n\int_{\mathbb{R}} K(x)  g(w_n(x,0))v(x,0)  \,\ud x  =  \int_{\mathbb{R}} K(x)  g(w(x,0)v(x,0))  \,\ud x , \,\,\,\,\text{for all $v \in E$} \label{Roma1}.
\end{align}
\end{proposition}
\begin{proof}
Let us prove \eqref{Veneto1} and \eqref{Toscana1}. Let $\sup_{n\in\N}\|w_n\|=:M\in (0,1)$.
By virtue of Lemma~\ref{boundmoser} there are $\alpha_M>\omega$ and $E>0$ with
\begin{equation}
\label{bound-crit}
\sup_{n\in\N}\int_{\R} (e^{\alpha_M w_n(x,0)^2}-1) \,\ud x\leq E,
\qquad  \int_\R |w_n(x,0)|^2 \,\ud x\leq E
\end{equation}
Let us fix $\eps>0$ and $q>2$. By virtue of  (g1) and (g2)$'$ we know that
$$
\limsup_{s\to +\infty} \frac{g(s)s}{e^{\alpha_M s^2}-1}=\limsup_{s\to +\infty} \frac{G(s)}{e^{\alpha_M s^2}-1}=0,
\qquad
\limsup_{s\to 0^+} \frac{g(s)s}{s^2}=\limsup_{s\to 0^+} \frac{G(s)}{s^2}=0.
$$
Then there exist $0<s_0(\eps)<s_1(\eps)$, $C(\eps)>0$  and $C_0$ depending only upon $K$, with
\begin{align}\label{2.12b}
 |K(x)G(s)|&\leq \epsilon C_0(|s|^2 +e^{\alpha_M s^2}-1)+ C(\eps) K(x)\chi_{[s_0(\eps),s_1(\eps)]}(|s|)|s|^q, \,\,\quad\text{for all $s\in \R$}, \\
\label{2.12.1b}
  |K(x)g(s)s|&\leq \epsilon C_0(|s|^2 +e^{\alpha_M s^2}-1)+ C(\eps) K(x)\chi_{[s_0(\eps),s_1(\eps)]}(|s|)|s|^q, \,\,\quad\text{for all $s\in \R$}.
\end{align}
Notice again that, by means of \eqref{decay-K}, there exists $r(\eps)>0$ such that
\begin{equation*}
 \int_{A^n_\eps \cap B^c_{r(\eps)}(0)} K(x) \,\ud x\leq \frac{\epsilon}{C(\eps)s_1(\eps)^q}, \quad\text{for all $n \in \N$}.
\end{equation*}
Now, combining the above inequality with \eqref{bound-crit} and  \eqref{2.12b}-\eqref{2.12.1b}, we have
\begin{align*}
& \int_{B^{c}_{r(\eps)}(0)} K(x)G(w_n(x,0)) \,\ud x\leq (2C_0E + 1)\epsilon ,\quad\text{for all $n \in \N$}, \\
& \int_{B^{c}_{r(\eps)}(0)} K(x)g(w_n(x,0))w_n(x,0) \,\ud x\leq (2C_0E + 1)\epsilon ,\quad\text{for all $n \in \N$}.
\end{align*}
The rest of the proof for \eqref{Veneto1} and \eqref{Toscana1} follows an in Proposition~\ref{convergeG1}, with $\alpha$
replaced by $\alpha_M$. 
Concerning \eqref{Roma1}, since $\alpha_M M^2<\omega$ there exists $m>1$ very close to $1$ such that $m\alpha_M M^2<\omega$.
Then $(K(x) g(w_n(x,0))\chi_{\{|w_n(x,0)|\geq 1\}})$ is bounded in $L^m(\R)$ by Lemma~\ref{boundmoser} since
$$
|K(x) g(w_n(x,0))\chi_{\{|w_n(x,0)|\geq 1\}}|^{m}\leq C(e^{m\alpha_M w_n(x,0)^2}-1).
$$
Then, for all $v\in E\subset L^{m'}(\R)$ (as $m'>2$), we get
\begin{equation*}
\lim_k\int_{\R} K(x) g(w_n(x,0))\chi_{\{|w_n(x,0)|\geq 1\}} v(x,0) \,\ud x=\int_{\R} K(x) g(w(x,0))\chi_{\{|w(x,0)|\geq 1\}} v(x,0) \,\ud x.
\end{equation*}
This concludes the proof of \eqref{Roma1}.
\end{proof}

\begin{proposition}
\label{C-uno}
$J\in C^1(E,\R)$. 
\end{proposition}
\begin{proof}
Let $(w_n)\subset E$ with $w_n\to w$ strongly in $E$. There exist $C>0$ and $\alpha>0$ such that 
$$
 |g(s)|^2\leq C(|s|^2 +e^{\alpha s^2}-1), \,\,\quad\text{for all $s\in \R$}.
$$
This choice fits both the subcritical and critical case. Hence
$$
|g(w_n(x,0))-g(w(x,0))|^2\leq C(|w_n(x,0)|^2 +e^{\alpha w_n(x,0)^2}-1)+C(|w(x,0)|^2 +e^{\alpha w(x,0)^2}-1).
$$
Taking into account Lemma~\ref{exp-conv}, we have
$$
\lim_n \int_{\R} (|w_n(x,0)|^2 +e^{\alpha w_n(x,0)^2}-1)\,\ud x=\int_{\R} (|w(x,0)|^2 +e^{\alpha w(x,0)^2}-1)\,\ud x.
$$
Then, by the Generalized Dominated Convergence Theorem, $\|g(w_n)-g(w)\|_{L^2}\to 0$. In turn,
\begin{align*}
\sup_{\|v\|\leq 1}\Big|\int_{\R} K(x)(g(w_n(x,0))-g(w(x,0)))v \,\ud x\Big|\leq C\|g(w_n)-g(w)\|_{L^2}\sup_{\|v\|\leq 1}\|v\|_{L^{2}(\R)}
\leq Co_n(1),
\end{align*}
which concludes the proof.
\end{proof}

\noindent
Next, we show that $J$ satisfies the Mountain Pass geometry.

\begin{lemma}
\label{MPG} The functional $J$ satisfies
\begin{enumerate}
\item There exists $\beta, \rho>0$ such that $J(w)\geq \beta$ if $w\in E$ and $\|w\|=\rho$;
\item There exists $e\in E\backslash\{0\}$ with $\|e\|>\rho$ such that $J(e)\leq 0$;
\end{enumerate}
\end{lemma}
\begin{proof}
Assertion (2) is straightforward due to the superquadraticity assumptions. For (1), let us consider
$w\in E$ with $\|w\|=\rho<1$ and $\omega<\alpha<\omega/\rho^2$. By the growth conditions
on $g$ (both critical and subcritical), there exist $r>1$ so close to $1$ that $r\alpha<\omega/\rho^2$, $q>2$ and $C>0$ with
$$
G(s)\leq \frac{1}{4}s^2+C(e^{r\alpha s^2}-1)^{1/r}s^q,\quad \text{for all $s\in\R^+$}.
$$
Then, taking into account inequality \eqref{chain-exp}, we have
\begin{align*}
J(w) &\geq \frac{1}{2}\|w\|^2-\frac{1}{4}\|w(\cdot,0)\|_{L^2}^2-C\int_{\R}\big(e^{r\alpha |w(x,0)|^2}-1\big)^{1/r}|w(x,0)|^q \,\ud x \\
& \geq \frac{1}{4}\|w\|^2-C\Big(\int_{\R}\big(e^{r\alpha |w(x,0)|^2}-1\big)\,\ud x\Big)^{1/r}  \Big(\int_\R |w(x,0)|^{r'q} \,\ud x)^{1/r'} \\
& \geq \frac{1}{4}\|w\|^2-C\|w\|^q=\frac{1}{4}\rho^2-C\rho^q=\beta>0,
\end{align*}
for every $\rho>0$ sufficiently small.
\end{proof}

\noindent
Therefore, there exists a sequence $ (w_n) \subset E$, so called {\it Cerami sequence} such that
\begin{equation}
\label{Cerami}
J(w_n)\to c, \quad (1 +\|w_n\|) \|J'(w_n)\|\to 0,
\end{equation}
where $c$  is given by
$$
c=\inf_{\gamma \in \Gamma}\max_{t\in[0,1]}J(\gamma(t)),
$$
with 
$$
\Gamma=\big\{\gamma \in C([0,1],E): \gamma(0)=0\ \mbox{and}\ J(\gamma(1))\leq 0\big\}.
$$

\noindent
We have the following result
\begin{lemma}
\label{cerami} 
The Cerami  sequence $(w_n) \subset E$ is bounded and $\|w_n^-\|\to 0$ as $n\to\infty$.
\end{lemma}
\begin{proof}
If $g$ has critical growth, the assertion is obvious since
the Ambrosetti-Rabinowitz condition (AR) is assumed. On the contrary, in the subcritical case, the proof follows by
mimicking the argument in the first part of the proof of \cite[Lemma 2.3]{JMO}, which is based upon monotonicity
of ${\mathcal H}(s)=sg(s)-2G(s)$,
holding since $\frac{g(s)}{s}$ is non-decreasing in $\R^+$,
and the application of \eqref{Veneto}.
\end{proof}

\noindent
To handle the case where $g$ is at critical growth, we shall need the following result.

\begin{lemma}
\label{boundmiuno}
Let $(w_n)\subset E$ be a bounded Palais-Smale sequence for the functional $J$ at the Mountain Pass energy level $c$.
Then 
$$
\sup_{n\in\N}\|w_n\|\in (0,1),
$$ 
provided that the constant $C_q>0$ which appears in {\rm (g3)$'$} is sufficiently large.
\end{lemma}
\begin{proof}
Let $q>2$ and $C_q>0$ as in assumption (g3)$'$. Fix $\psi\in C^\infty_c(\R)\setminus\{0\}$ and let us denote
$$
L:=\inf_{{\rm supt}(\psi)}K>0,  \,\,\,\quad 
\S_q:=\frac{\Big(\displaystyle \int_{\R^2}(\nabla \psi|^2  \,\ud x \ud y + \int_{\R} \psi(x,0)^2  \,\ud x\Big)^{1/2}}
{\Big(\displaystyle\int_{\R}|\psi(x,0)|^q \,\ud x\Big)^{1/q}}=\frac{\|\psi\|}{\|\psi(\cdot,0)\|_{L^q}}.
$$
Also, let $\omega_q>0$ be such that $J(\omega_q \psi)<0$. This is possible since
\begin{align*}
J(\omega \psi)=\frac{\omega^2}{2}\|\psi\|^2-\int_\R K(x)G(\omega\psi(x,0)) \,\ud x
\leq\frac{\omega^2}{2}\|\psi\|^2-\omega^q C_qL  \int_{\R}| \psi(x,0)|^q  \,\ud x<0,
\end{align*}
for every $\omega=\omega_q>0$ sufficiently large. Then, $\gamma\in C([0,1],E)$ defined by $\gamma(t):=t\omega_q\psi\in \Gamma$
for  $t\in [0,1]$ belongs to the class of continuous paths $\Gamma$. Hence, we get
\begin{align*}
c&=\inf_{\gamma \in \Gamma}\max_{t\in[0,1]}J(\gamma(t)) 
\leq\max_{t\in [0,1]}J(t\omega_q \psi)\leq\max_{t\in\R^+}J(t\psi)\\
&=\max_{t\geq 0} \Big(\frac{t^2}{2}\|\psi\|^2 -C_qL  t^q\|\psi(\cdot,0)\|^q_{L^q} \Big) \\
&=\max_{t\geq 0} \Big(\frac{\S_q^2}{2} t^2 \|\psi(\cdot,0)\|^2_{L^q} -C_qL t^q \|\psi(\cdot,0)\|^q_{L^q} \Big) \\
&=\max_{t\geq 0} \Big(\frac{\S_q^2}{2} t^2-C_qL t^q\Big) 
= \frac{q-2}{2q}\frac{\S^{\frac{2q}{q-2}}_q}{(qC_qL)^{\frac{2}{q-2}}}.
\end{align*}
On the other hand, since $(w_n)$ is a Palais-Smale sequence, we get
\begin{align*}
c = \limsup_n \big (J(w_n)- \frac{1}{\theta} J'(w_n)(w_n) \big)
\geq \frac{\theta-2}{2 \theta}\limsup_{n} \|w_n\|^2.
\end{align*}
In turn, by combining the above inequalities, we get
\begin{equation*}
\limsup_{n} \|w_n\|^2\leq \frac{2\theta}{\theta-2}\frac{q-2}{2q}\frac{\S^{\frac{2q}{q-2}}_q}{(qC_qL)^{\frac{2}{q-2}}}<1,
\end{equation*}
provided that $C_q$ is large enough.
\end{proof}

\section{Proof of the main results}
\subsection{Proof of Theorem~\ref{mainsub} completed}
\label{prova1}
\noindent
In light of Lemma~\ref{MPG}, there exists a {\em Cerami} sequence $\{w_n\}\subset E$ for $J$ at the Mountain Pass level $c>0$. 
From Lemma~\ref{cerami} it follows that $\{w_n\}$ is bounded, $w^-_n\to 0$ in $E$, and thus it admits 
a nonnegative weak limit $w\in E$. 
By \eqref{Roma} of Proposition \ref{convergeG1},  it follows that
\begin{equation}
\label{solve}
 \int_{ \R^{2}_{+}}\nabla w \cdot\nabla v  \,\ud x \ud y + \int_{\R}w(x,0)v(x,0)  \,\ud x = \int_{\mathbb{R}} K(x) g(w(x,0))v(x,0)  \,\ud x,
 \quad \forall v\in E.
\end{equation}
Then, we have a weak solution $u \in H^{1/2}(\R)$ to \eqref{PS}.
We have $u>0$ if $u\neq 0$, arguing as in \cite{JMO}. We prove that $w=E_s(u)\not\equiv 0$.
In fact, $(w_n)$ converges  to $w$ strongly in $E,$ as $n\to\infty$.
Indeed, since $J'(w_n)(w_n)=o_n(1),$ we have, by \eqref{Toscana} of Proposition \ref{convergeG1},
\begin{equation*}
\lim_n \|w_n\|^2  =\lim_n \int_{\R}K(x)g(w_n(x,0))w_n(x,0) \,\ud x 
=\int_{\R}K(x)g(w(x,0))w(x,0) \,\ud x=\|w\|^2,
\end{equation*}
that is, $ w_n \rightarrow w$ in $E.$  Hence $J(w)=c>0$ by continuity, yielding $w\not\equiv 0.$ \qed

\subsection{Proof of Theorem~\ref{main} completed}
\label{prova2}
\noindent 
In light of Lemma~\ref{MPG}, there exists a {\em Cerami} sequence $\{w_n\}\subset E$ for $J$ at the Mountain Pass level $c>0$. 
From Lemma~\ref{cerami} it follows that $\{w_n\}$ is bounded, $w^-_n\to 0$ in $E$, and thus it admits 
a nonnegative weak limit $w\in E$.  
By taking $C_q$ sufficiently large in assumption (g3)$'$, in light of Lemma~\ref{boundmiuno}, it follows that
$\sup_{n\in\N}\|w_n\|\in (0,1)$. Then, we are allowed to apply the assertions of Proposition \ref{convergeG2}.
By \eqref{Roma1} of Proposition \ref{convergeG2},  it follows that \eqref{solve} is satisfied.
Then, we have a weak solution $u \in H^{1/2}(\R)$ to \eqref{PS}.
We have $u>0$ if $u\neq 0$, arguing as in \cite{JMO}. Indeed $u\neq 0$. In fact $w=E_s(u)\not\equiv 0$.  
Suppose by contradiction that $w=0.$
Then, since $w_n \rightharpoonup 0$, we have by \eqref{Veneto1} and \eqref{Toscana1} of Proposition \ref{convergeG2}
$$
\lim_n\int_{\R}K(x)G(w_n(x,0) \,\ud x=\lim_n\int_{\R}K(x) g(w_n(x,0)w_n(x,0) \,\ud x=0.
$$
Then, from
\begin{align*}
& \frac{1}{2}\|w_n\|^2 - \int_{\mathbb{R}} K(x) G(w_n(x,0))   \,\ud x  
=c+o_n(1),  \\
& \|w_n\|^2 - \int_{\mathbb{R}} K(x) g(w_n(x,0))w_n(x,0)  \,\ud x =o_n(1)
\end{align*}
we get a contradiction, since $c>0$. The proof is complete. \qed

\medskip

\end{document}